\documentclass[11pt]{amsart}

\usepackage{amssymb,amsthm,amsfonts,latexsym,mathtools}
\usepackage{amsmath}
\usepackage{color}
\usepackage{hyperref}
\usepackage[capitalize]{cleveref}
\usepackage{enumitem}
\usepackage{verbatim}
\usepackage[all]{xy}

\allowdisplaybreaks

  \makeatletter
  \newcommand{\numberlike}[2]{%
     \expandafter\def\csname c@#1\endcsname{%
         \expandafter\csname c@#2\endcsname}%
  }
  \makeatother

  \def\DefaultNumberTheoremWithin{section}

  \theoremstyle{plain}
  \newtheorem{lemma}{Lemma}
     \numberwithin{lemma}{\DefaultNumberTheoremWithin}
     \labelformat{lemma}{Lemma~#1}
  \newtheorem{theorem}{Theorem}
     \numberwithin{theorem}{\DefaultNumberTheoremWithin}
     \numberlike{theorem}{lemma}
     \labelformat{theorem}{Theorem~#1}
  \newtheorem*{maintheorem*}{Corollary of Main Theorem}
  \newtheorem{corollary}{Corollary}
     \numberwithin{corollary}{\DefaultNumberTheoremWithin}
     \numberlike{corollary}{lemma}
     \labelformat{corollary}{Corollary~#1}
  \newtheorem{proposition}{Proposition}
     \numberwithin{proposition}{\DefaultNumberTheoremWithin}
     \numberlike{proposition}{lemma}
     \labelformat{proposition}{Proposition~#1}
  
     \numberwithin{conjecture}{\DefaultNumberTheoremWithin}
     \numberlike{conjecture}{lemma}
     \labelformat{conjecture}{Conjecture~#1}

  \theoremstyle{definition}
  \newtheorem{definition}{Definition}
     \numberwithin{definition}{\DefaultNumberTheoremWithin}
     \numberlike{definition}{lemma}
     \labelformat{definition}{Definition~#1}
  
     \numberwithin{question}{\DefaultNumberTheoremWithin}
     \numberlike{question}{lemma}
     \labelformat{question}{Question~#1}
  
     \numberwithin{problem}{\DefaultNumberTheoremWithin}
     \numberlike{problem}{lemma}
     \labelformat{problem}{Problem~#1}
     
  \theoremstyle{remark}
  
     \numberwithin{remark}{\DefaultNumberTheoremWithin}
     \numberlike{remark}{lemma}
     \labelformat{remark}{Remark~#1}
  \newtheorem{example}{Example}
     \numberwithin{example}{\DefaultNumberTheoremWithin}
     \numberlike{example}{lemma}
     \labelformat{example}{Example~#1}

\labelformat{figure}{Figure~#1}
\labelformat{chapter}{Chapter~#1}
\labelformat{appendix}{Appendix~#1}
\labelformat{section}{Section~#1}
\labelformat{subsection}{Subsection~#1}

\def\CC{\mathbb{C}}
\def\KK{\mathbb{K}}
\def\ZZ{\mathbb{Z}}

\def\D{\mathcal{D}}
\def\F{\mathcal{F}}
\def\I{\mathcal{I}}
\def\L{\mathcal{L}}
\renewcommand{\S}{\mathcal{S}}
\def\T{\mathcal{T}}

\newcommand{\GF}[1]{\mathbb{F}_{#1}}

\newcommand{\tq}{\mathrel{{\ensuremath{\: : \: }}}}
\DeclareMathOperator{\rk}{rk}
\DeclareMathOperator{\Id}{\mathrm{Id}}
\DeclareMathOperator{\op}{op}
\DeclareMathOperator{\Cl}{Cl}
\DeclareMathOperator{\CB}{CB}
\DeclareMathOperator{\GL}{GL}
\DeclareMathOperator{\bool}{\Sigma} 
\newcommand\PD{\mathcal{PD}}
\DeclareMathOperator\FC{FC} 
\newcommand{\red}[1]{#1^*}
\newcommand{\gen}[1]{\left\langle #1\right\rangle}

\usepackage{ifthen}

\renewcommand{\max}[1][]{%
   \ifthenelse{ \equal{#1}{} }
      {\mathrm{max}}
      {\mathrm{max}_{#1}}
}
\renewcommand{\min}[1][]{%
   \ifthenelse{ \equal{#1}{} }
      {\mathrm{min}}
      {\mathrm{min}_{#1}}
}

\title[The common basis complex and the partial decomposition poset]{The common basis complex and the partial decomposition poset}

\author{Benjamin Br{\"u}ck}
\address{Universit{\"a}t M{\"u}nster \\
Institut f{\"u}r Mathematische Logik und Grundlagenforschung \\
48149 M{\=u}nster, Germany}
\email{benjamin.brueck@uni-muenster.de}

\author{Kevin I. Piterman}
\address{Philipps-Universit\"at Marburg \\
Fachbereich Mathematik und Informatik \\
35032 Marburg, Germany}
\email{piterman@mathematik.uni-marburg.de}

\author{Volkmar Welker}
\address{Philipps-Universit\"at Marburg \\
Fachbereich Mathematik und Informatik \\
35032 Marburg, Germany}
\email{welker@mathematik.uni-marburg.de}


\begin{document}

\begin{abstract}
For a finite-dimensional vector space $V$, the common basis complex of $V$ is the simplicial complex whose vertices are the proper non-zero subspaces of $V$, and $\sigma$ is a simplex if and only if there exists a basis $B$ of $V$ that contains a basis of $S$ for all $S\in \sigma$.
This complex was introduced by Rognes in 1992 in connection with stable buildings.

In this article, we prove that the common basis complex is homotopy equivalent to the proper part of the poset of partial direct sum decompositions of $V$. Moreover, we establish this result in a more general combinatorial context, including the case of free groups, matroids, vector spaces with non-degenerate sesquilinear forms, and free modules over commutative Hermite rings, such as local rings or Dedekind domains.
\end{abstract}

\maketitle

\section{Introduction}
\label{sec:introduction}

In \cite{Rognes}, Rognes introduced the complex of common bases of a free module of rank $n$ over a ring $R$ with the invariant dimension property.
For $M = R^n$, the common basis complex $\CB(M)$ has as vertices the proper non-zero free summands of $M$, and there is a simplex $\{M_1,\ldots,M_r\}$ if and only if there exists a basis $B$ of $M$ that contains a basis of $M_i$ for all $1\leq i\leq r$.
Note that the common basis complex has dimension $2^n-3$.

The topological suspension of this complex corresponds to the \textit{stable building} of $R^n$ as defined by Rognes.
Its homology groups show up in a spectral sequence associated to Rognes' filtration of the algebraic K-theory spectrum $K(R)$.
Rognes showed that these homology groups vanish in degrees larger than $2n-2$, and conjectured that, when $R$ is local or Euclidean, the stable building should be $(2n-3)$-connected (so we find reduced homology only in degree $2n-2$). In other words, he conjectured that the common basis complex $\CB(M)$ should be $(2n-4)$-connected for a free $R$-module $M$ of rank $n$. 
If this was true, it would simplify the computation of $K(R)$ using Rognes' filtration (see \cite[p.~836]{Rognes}).

In this article, we relate the common basis complex to the poset of partial direct sum decompositions into free summands.
To be more precise, suppose $M$ is a free $R$-module of finite rank.
A \textit{partial direct sum decomposition} of $M$ is a subset $\{N_{i_1},\ldots,N_{i_k}\}$ of non-zero free summands belonging to a full decomposition $M = N_1\oplus \cdots \oplus N_r$.
We denote the collection of all such partial decompositions with order given by refinement by $\PD(M)$, and let $\red{\PD(M)} = \PD(M)\setminus \{\emptyset, \{M\}\}$ denote the proper part of this poset.
As a consequence of our main result \ref{thm:main} we get:

\begin{corollary}
\label{thm_hom_eq_cb_pd}
If $R$ is a commutative ring such that every stably free module is free, then $\CB(R^n)$ is $\GL_n(R)$-homotopy equivalent to $\red{\PD(R^n)}$.
\end{corollary}

Since $\red{\PD(R^ n)}$ is $(2n-3)$-dimensional, 
and local rings and Euclidean rings satisfy the hypothesis of this corollary,
we get the following
restatement of Rognes' conjecture:

\vspace{0.2cm}

\noindent
\textbf{Rognes' Conjecture.} If $R$ is local or Euclidean, then $\red{\PD(R^n)}$ is spherical.

\vspace{0.2cm}

The conjecture was recently established by Miller, Patzt and Wilson \cite{MPW} when $R$ is a field. Here, \ref{thm_hom_eq_cb_pd} yields:

\begin{corollary}
    \label{cor:field}
    For a field $R$, the complex $\CB(R^ n)$ and the poset $\red{\PD(R^n)}$ are $\GL_n(R)$-homotopy equivalent,
    and they have the homotopy type of a wedge of
    spheres of dimension $2n-3$.
    In particular, the free abelian groups $\widetilde{H}_{2n-3}(\CB(R^ n);\ZZ)$ and $\widetilde{H}_{2n-3}(\red{\PD(R^n)};\ZZ)$ are isomorphic as $\GL_n(R)$-modules.
\end{corollary}

In the case that $R$ is a finite field, the $\GL_n(R)$-module structure of $\widetilde{H}_{2n-3}(\CB(R^ n);\CC)\cong \widetilde{H}_{2n-3}(\red{\PD(R^n)};\CC)$ is known: 
unpublished work of Hanlon, Hersh and Shareshian \cite{HHS} describes it explicitly, see Eq. (\ref{eq:rep}) in \ref{sec_ex_free_modules} for details. Additionally, they show that $\red{\PD(\GF{q}^n)}$ is not only spherical but Cohen-Macaulay.

To the best of our knowledge, the fact that for general fields $R$ the complex $\CB(R^ n)$ has the homotopy type of a wedge of spheres is 
not stated in the literature. 
However, Jeremy Miller pointed out to us 
that it actually follows from the results of Rognes \cite{Rognes23} and Miller--Patzt--Wilson \cite{MPW}: as explained in \cite[Section 10]{Rognes23},
the homology of $\CB(R^n)$ arises from a free chain complex where the homology group of degree $2n-3$ is free since it is the kernel of the last map of this chain complex.
Hence, if homology is concentrated in a single dimension, as was shown for $R$ a field in \cite{MPW}, the complex $\CB(R^ n)$ is homotopy equivalent to a wedge of spheres of that dimension.
Jeremy Miller also informed us about an alternative proof that $\CB(R^ n)$ 
has the homotopy type of a $(2n-3)$-dimensional simplicial complex.

\ref{thm:main} also implies results analogous to \ref{thm_hom_eq_cb_pd} in the following
situations:

\begin{enumerate}
    \item If $M$ is a matroid, then a partial decomposition is a non-empty subset of a direct sum decomposition of $M$ into proper flats.
    \item If $V$ is a vector space with a non-degenerate $\sigma$-sesquilinear form, one can define a partial decomposition as a non-empty subset of a direct sum decomposition of $V$ into pairwise orthogonal non-degenerate non-zero proper subspaces.
    \item Similarly, if $V$ is a non-degenerate unitary or symplectic space, then we can also consider direct sum decompositions into non-zero totally isotropic subspaces.
    \item If $F_n$ is the free group of rank $n$, then a partial decomposition is a set of non-trivial proper free factors $\{H_1,\ldots,H_k\}$ such that $F_n = H_1*\ldots *H_k*K$ for some free factor $K$.
\end{enumerate}

In these cases, we can also define a common basis complex, which naturally generalizes the case of free modules.
The general definitions can be found in \ref{sec:preliminaries} and are spelled out for the cases above in \ref{sec:examples}. 
We prove the following:

\begin{corollary} \label{cor:remainingcases}
  Let $X$ be a matroid, a finite-dimensional vector space with a non-de\-gen\-erate $\sigma$-sesquilinear form, or a free group of finite rank.
  Then the common basis complex $\CB(X)$ is homotopy equivalent to the poset $\red{\PD(X)}$ of partial decompositions ordered by refinement.
\end{corollary}

Moreover, if $X$ comes with a natural group $G$  of automorphisms acting on $\CB(X)$ and $\red{\PD(X)}$, then the homotopy equivalence above is $G$-equivariant.
This is the case for example for the isometry group of a vector space with a non-degenerate form acting on the posets of non-degenerate subspaces and totally isotropic subspaces, or for the automorphism group of the free group acting on the poset of free factors.
See \ref{thm:main} for more details.

If $X$ is one of the objects above, we may ask whether $\red{\PD(X)}$ is spherical, that is,~a wedge of spheres of the maximal possible dimension.
This would imply that Rognes' conjecture, as stated above, also holds in these other contexts.
However, this is not the case in general.
For instance, if $M$ is the uniform matroid of rank $k$ on $n$ points, with $k\leq n-1$, then $\red{\PD(M)}$ is homotopy equivalent to a non-trivial wedge of spheres of dimension $k-1$, which is smaller than $\dim(\red{\PD(M)}) = 2k-3$ for $k\geq 3$.
A similar situation arises if $X$ is a finite-dimensional vector space $V$ with a non-degenerate Hermitian, orthogonal or symplectic form, where the poset of partial orthogonal decompositions into non-degenerate subspaces is homotopy equivalent to the poset of proper orthogonal decompositions of $V$, which has smaller dimension and is not contractible in general.
See \ref{sec:examples} for more details on these examples.

\bigskip

\textbf{Acknowledgments.}
The first author was supported by the Deutsche Forschungsgemeinschaft [Project-ID 427320536 -- SFB 1442], as well as by Germany’s Excellence Strategy [EXC 2044 -- 390685587].
The second author was supported by the Alexander von Humboldt Stiftung and the Fonds voor Wetenschappelijk Onderzoek – Vlaanderen [grant 12K1223N].

We thank Aldo Conca, Jeremy Miller and Jennifer C.~H.~Wilson for helpful comments and suggestions.
We are also grateful to an anonymous referee for insightful comments that helped to improve the presentation of this article.

\section{Preliminaries and Statement of the Main Result}
\label{sec:preliminaries}

In this section, we present the notation and definitions we will use throughout this paper. 
The definitions around decompositions and frames are adapted from \cite{PW} 
avoiding the more technical categorical approach taken there.
We also state our main result, \ref{thm:main}, at the end of this section.

\subsection{Definitions of complexes and posets}

Let $\S$ be a (not necessarily finite) poset.
We write $\check{\S}$ for the poset $\S \cup \{0_\S,1_\S\}$
where $0_\S < x < 1_\S$ for all $x \in \S$.
For $x\in \S$, we define the subposet $\S_{\leq x} = \big\{\,y\in \S\,\tq\, y\leq x\,\big\}$.
Analogously defined are the subposets $\S_{\geq x}$, $\S_{<x}$ and $\S_{> x}$.
The order complex of a poset $\S$ is the simplicial complex $\Delta \S$ with simplices the linearly ordered subsets of $\S$.
When we speak of posets in terms of topological spaces we identify the poset with its order complex.
The height of an element $x\in \S$ is the dimension of the order complex of $\check{\S}_{\leq x}$, and we denote it by $h(x)$.
The height $h(\S)$ of $\S$ is the height of $1_\S$ in $\check{\S}$ or, equivalently, the dimension of the order complex of $\check{\S}$. 
For $x , y \in \S$, if the subposet $\S_{\geq x} \cap \S_{\geq y}$ has
a unique minimal element $z$ then we call $z$ the join of $x$ and $y$ and write $x \vee y$ for $z$. 
If the subposet $\S_{\leq x} \cap \S_{\leq y}$ has a unique maximal element
$z$ then we call $z$ the meet of $x$ and $y$ and write 
$x \wedge y$ for $z$. The meet is associative, that is,~if $x,y,z \in \S$ and both $x\wedge y$ and $x \wedge y \wedge z$ exist, then
\begin{equation}
    \label{eq_associativity_join}
    (x\wedge y) \wedge z = x \wedge y \wedge z = x\wedge (y \wedge z).
\end{equation}
Finally, a subset of $\S$ of pairwise
incomparable elements is called an antichain.

\begin{definition}
\label{def_frame}
  Let $\S$ be a poset.
  For a finite non-empty subset $\tau \subseteq\S$, let $\bool(\tau)$ be the subposet of all joins of non-empty subsets of $\tau$ that exist in $\S$.
  We say that $\tau$ is a \textit{frame} if it is an antichain such that all subsets of $\bool(\tau)$ that have a lower bound in $\S$ have a meet in $\S$ and this meet lies in $\bool(\tau)$,
  \begin{equation}
  \label{eq_def_frame}
      \forall \sigma\subseteq \bool(\tau):
      \left[\exists x\in \S :  x\leq y \,\forall y \in \sigma  \implies \bigwedge_{y\in\sigma} y \in  \Sigma(\tau) \right].
  \end{equation}
\end{definition}

Note that this condition in particular implies that for a frame $\tau$, the subposet $\bool(\tau)$ is the proper part of a lattice and closed under taking (existing) meets in $\S$ (take $x = \bigwedge_{y\in\sigma} y$ in Eq. (\ref{eq_def_frame})). The condition also
implies that no two different elements of a frame $\tau$ have a lower bound in $\S$.
In most of our examples, including the setup of Rognes' conjecture, $\tau$ will be a set of minimal elements of $\S$ and $\bool(\tau)$ will be isomorphic to the poset of non-empty faces of the boundary of a simplex. Since these assumptions do not simplify our proof and for potential future applications, we formulate our result using the general definition above. For an example where $\bool(\tau)$ is not
the poset of proper faces of a simplex, see \ref{sec_ex_isotropic_subspaces}.

We think of frames as variants of ``bases'' or ``building blocks" for the elements of $\S$. To describe subsets of $\S$ that are ``generated'' by such a basis, we consider the following properties for $\sigma\subseteq \S$ and $\tau$ a frame:
\begin{enumerate}[label={$P_{\theenumi}(\sigma,\tau)$:\,}, leftmargin=2.3cm]
\item \label{item_span_of_frame} For all $y\in \sigma$, we have $y\in \bool(\tau)$.
\item \label{item_unique_containment} For all $x\in \tau$, there is at most one $y\in \sigma$ such that $x\leq y$.
\item \label{item_complete_decomp} For all $x\in \tau$, there is at least one $y\in \sigma$ such that $x\leq y$.
\end{enumerate}
\noindent
We trivially see that $P_1(\tau,\tau)$ and $P_3(\tau,\tau)$ hold. Also $P_2(\tau,\tau)$ holds since $\tau$ is an antichain.
  
\begin{definition}
\label{def_decomp_complexes}
Let $\F$ be a non-empty set of frames in $\S$. We say that $\sigma\subseteq \S$ is:
\begin{itemize}
\item \emph{basis compatible} if there is $\tau\in \F$ such that $P_1(\sigma,\tau)$ holds;
\item a \emph{partial decomposition} if there is $\tau\in \F$ such that $P_1(\sigma,\tau)$ and $P_2(\sigma,\tau)$ hold;
\item a \emph{decomposition} if there is $\tau\in \F$ such that $P_1(\sigma,\tau)$, $P_2(\sigma,\tau)$ and $P_3(\sigma,\tau)$ hold.
\end{itemize}
\end{definition}

We write $\CB(\S,\F)$ for the poset of all $\sigma \subseteq \S$ that are basis compatible together with the partial ordering given by containment.
We call $\CB(\S,\F)$ the \emph{common basis complex} (of $\S$ with respect to $\F$). It specializes to the common basis complex by Rognes \cite{Rognes} described in \ref{sec:introduction} when $\S$ is the poset of non-zero proper summands of $R^n$ and $\F$ the collection of direct sum decompositions of $R^n$ into rank one free summands (see \ref{sec:examples}).

We write $\red{\PD(\S,\F)}$ and $\red{\D(\S,\F)}$ for the posets of all non-empty subsets $\sigma \subseteq \S$ that are partial decompositions and decompositions, respectively, ordered by refinement.
That is, $\sigma\leq \sigma'$ if and only if for all $x\in \sigma$ there is $y\in \sigma'$ such that $x\leq y$ in $\S$.

The following lemma summarizes some elementary properties of these posets.

\begin{lemma}
Let $\S$ be a poset and $\F$ a non-empty set of frames in $\S$.
    \begin{enumerate}
        \item The elements of $\red{\PD(\S,\F)}$ are exactly the subsets of elements in $\red{\D(\S,\F)}$.
        \item If $\tau\in \F$ and $\nu\subseteq \tau$ then $\nu\in \CB(\S,\F)$, $\nu\in \red{\PD(\S,\F)}$, and $\tau \in \red{\D(\S,\F)}$.
        \item $\CB(\S,\F)$ is a simplicial complex.
        \item If $m \coloneqq \sup \big\{\, |\tau| \, \tq \,\tau\in \F\,\big\}$ is finite, then $$m-1\leq \dim\big(\,\CB(\S,\F)\,\big) = \sup \big\{\,|\bool(\tau)| \,\tq\, \tau \in \F\, \big\}-1 \leq 2^m-2.$$
        \item If for all $\tau\in \F$, the poset $\bool(\tau)$ is the poset of non-empty faces of the boundary of an $(m-1)$-simplex, then $\dim(\CB(\S,\F)) = 2^m-3$.
    \end{enumerate}
\end{lemma}

In general, the dimensions of $\red{\D(\S,\F)}$ and $\red{\PD(\S,\F)}$ depend heavily on the structure of $\S$. We give bounds on these under an additional assumption in \ref{lem_dimensions}.

\begin{example}
\label{ex_Zn}
Let $\S(\ZZ^n)$ be the poset of non-trivial, proper direct summands of $\ZZ^n$. This poset has height $n$.
Let $\F$ be the set of all collections $\{L_1, \ldots, L_n\}$ of lines that form a direct sum decomposition $\ZZ^n = L_1\oplus  \cdots \oplus L_n$.
For all $\tau\in \F$, the poset $\Sigma(\tau)$ is isomorphic to the poset of non-empty faces of an $(n-1)$-simplex, given by all spans of non-empty proper subsets of $\tau$.
Note that not all collections of $n$ linearly independent lines satisfy this condition. For example $L_1 = \gen{(1,0)}$ and $L_2 = \gen{(1,2)}$ are linearly independent lines in $\ZZ^2$, $L_1\vee L_2 = \ZZ^2$ in $\check{\S}(\ZZ^2)$ but $L_1 + L_2 = \gen{(1,0),(0,2)}\neq \ZZ^2$.

As each line in $\ZZ^n$ only contains two elements that span it, a frame in $\F$ is equivalent to the choice of a basis of the free abelian group $\ZZ^n$ modulo signs.
A subset $\sigma = \{M_1, \ldots, M_k\} \subset \S(\ZZ^n)$ lies in $\CB(\S(\ZZ^n),\F)$ if there is basis $B$ of $\ZZ^n$ such that each $M_i$ is spanned by a subset of $B$.
This set $\sigma$ is contained in $\red{\PD(\S(\ZZ^n),\F)}$ if furthermore, the $M_i$ have pairwise trivial intersection.
Finally, it lies in $\red{\D(\S(\ZZ^n),\F)}$ if $\ZZ^n = M_1\oplus  \cdots \oplus M_k$. 

If $b_1, \ldots, b_n$ is a basis of $\ZZ^n$, a maximal chain in $\red{\PD(\S,\F)}$ is given by
\begin{multline*}
   \{\gen{b_1}\} < \cdots < \{\gen{b_1}, \ldots, \gen{b_{n-1}}\}\\
    < \{\gen{b_1}, \ldots, \gen{b_n}\} < \{\gen{b_1, b_2}, \gen{b_3}, \ldots, \gen{b_n}\}< 
    \cdots <\{\gen{b_1, \ldots, b_{n-1}}\}. 
\end{multline*}
The last $n-1$ terms of this chain (in the bottom row) provide a maximal chain in $\red{\D(\S,\F)}$. There are $2^n-2$ non-empty proper subsets of $\{b_1, \ldots, b_n\}$, and the set of spans of all these subsets gives a maximal simplex in $\CB(\S,\F)$.
\end{example}

\subsection{Relation to the definitions in \texorpdfstring{\cite{PW}}{PW}}
\label{rk_def_PW}
In the following, we explain how our definition of frames and decompositions in key situations specializes to the definition given in \cite{PW}.
In particular, we show that elements of partial decompositions have ``additive height".
This recovers the dimension-theorem in the subspace poset of a vector space, where the dimension of a partial (direct sum) decomposition is the sum of the dimensions.

  In the previous work \cite{PW}, there is a notion of frame that we claim is a particular case of a family of frames $\F$ here.
  Let $\S$ be any poset, and suppose the order complex of $\check{\S}$ has dimension $n$.
  We take $\F$ to be the collection of frames $\tau \subseteq \S$, as defined in \ref{def_frame}, such that $|\tau|=n$ and $\Sigma(\tau)$ is isomorphic to the proper part of the Boolean lattice on $\tau$.
  In view of Definition 2.1 of \cite{PW}, we claim that $\red{\D(\S,\F)}$ is the upward closed subposet of $\D(\check{\S}) \setminus \{ \{1_{\S}\} \}$ whose minimal elements are the decompositions of size $n$.
  We have a similar description for $\red{\PD(\S,\F)}$.
  
  To prove this assertion, it only remains to verify the height condition, namely, that $h(\bigvee_{x\in \sigma}x) = \sum_{x\in \sigma} h(x)$ for any partial decomposition $\sigma$.
  Here $h$ is the height function for the poset $\S$ (observe that if $x\in \S$ then $h(x)$ equals the height of $x$ in $\check{\S}$ as defined in \cite{PW}).
  We prove this first for our set of frames and then show that this easily extends to partial decompositions.
  
  Write $\tau = \{x_1,\ldots,x_n\} \in \F$, and $z_i = x_1\vee\ldots \vee x_i \in \check{\S}$, $i\leq n-1$.
  Set $h(1_{\S}) = n$ and $h(0_{\S}) = 0$.
  Then we have that $z_0 = 0_\S$ and $z_i < z_{i+1}$ for all $i\geq 0$.
  Hence $h(z_i) + 1 \leq h(z_{i+1})$, and inductively we have a chain of inequalities
  \[ n-1 \leq h(z_1) + (n-2) \leq \cdots \leq h(z_{i}) + (n-i-1) \leq h(z_{i+1}) + (n-i-2) \leq \cdots \leq h(z_{n-1}) \leq n-1.\]
  We conclude that these are equalities and $h(z_i) = i$ for all $i$.
  Also $x_1\vee\ldots\vee x_n$ exists in $\check{\S}$ and is equal to $1_{\S}$.
  Therefore the elements of $\tau$ are minimal elements of $\S$ and for any $\tau'\subseteq \tau$ we have
  \[ h\big( \bigvee_{x\in\tau'} x \big) = |\tau'|. \]

  On the other hand, let $\sigma \in \red{\D(\S,\F)}$ be such that $P_i(\sigma,\tau)$ holds for all $i=1,2,3$.
  Put $\tau_y = \{ x\in \tau \tq x\leq y\}$ for $y\in \sigma$.
  Then $\{\tau_y\tq y\in \sigma\}$ is indeed a partition of $\tau$, and for all $y\in \sigma$ we have $y = \bigvee_{x\in \tau_y}x$.
  From the computation on the heights above, for $y\in \sigma$ we get
  \[ h(y) = h\big( \bigvee_{x\in \tau_y}x \big) = |\tau_y|.\]
  In particular, if $y_1,\ldots,y_r\in \sigma$ are distinct elements then
  \[ h(y_1\vee\ldots\vee y_r)  = \sum_{i=1}^r |\tau_{y_i}| = \sum_{i=1}^r h(y_i).\]
  This shows that the definitions presented here coincide with that given in Definition 2.1 of \cite{PW}, up to a selection of $n$-frames $\tau$ such that $\Sigma(\tau)$ is the proper part of a Boolean lattice on $n$ elements.
  In particular, from Lemma 2.5 of \cite{PW} we see that if $\sigma\in \red{\D(\S,\F)}_{\geq \tau}$ for some $\tau\in \F$ then $\sigma$ is obtained by taking joins of elements of $\tau$ and hence $\red{\D(\S,\F)}_{> \tau}$ is isomorphic to the proper part of the partition lattice on $\tau$.

\subsection{The extension property (EP)}

To prove our main theorem, we need the following property.

\begin{definition}
  We say that a non-empty collection of frames $\F$ has the extension property (EP) if for all $\sigma, \sigma'\in \red{\PD(\S,\F)}$ with $\sigma\leq\sigma'$, there is $\tau\in \F$ such that both $P_{1}(\sigma,\tau)$ and $P_{1}(\sigma',\tau)$ hold.
\end{definition}

This property roughly encodes the basis extension property that holds in many contexts: if we have a basis of a subspace, then we can extend it to the full space.
For example, this holds for $\ZZ^n$, as we show in the next example.

\begin{example}[\ref{ex_Zn} continued]
\label{ex_EP}
In the context of $\S(\ZZ^n)$, (EP) says that if $\sigma$ and $\sigma'$ are elements in $\red{\PD(\S(\ZZ^n),\F)}$ such that each $M\in \sigma$ is contained in some $N \in \sigma'$, then there is a line decomposition of $\ZZ^n$ such that each $M\in \sigma$ and $N\in \sigma'$ is spanned by a subset of these lines. This is true because of the following: if $N$ is a direct summand of $\ZZ^n$ containing $M_{1}, \ldots, M_{k}$ and the latter span a direct summand $M_1 \oplus \cdots \oplus M_k$ of $\ZZ^n$, then they also span a direct summand of $N$. Hence, if one chooses a basis $B_i$ for each $M_{i}$, then the union $\bigcup_{i=1}^k B_i$ can be extended to a basis of $N$. This implies (EP).
\end{example}

In fact, (EP) is also valid for chains of partial decompositions:

\begin{lemma}
\label{lm:EPproperty}
  Assume that $\F$ satisfies (EP) and that $\sigma_1 <\cdots < \sigma_t$ is a chain of elements from 
  $\red{\PD(\S,\F)}$.
  Then there is $\tau\in \F$ such that $P_1(\sigma_i,\tau)$ holds for all $1\leq i\leq t$. 
\end{lemma}
\begin{proof}
  We argue by induction on $t$.
  The claim is obvious for $t=1,2$.
  Suppose now that there is $\tau'\in \F$ such that $P_1(\sigma_i,\tau')$ holds for all $1\leq i \leq t-1$, that is,~every $x\in \bigcup_{i=1}^{t-1} \sigma_i$ is a join of elements from $\tau'$.
  Let $\nu \coloneqq \{y\in \tau'\tq \{y\}\leq \sigma_{t-1} \}$.
  Clearly, $\nu\in \red{\PD(\S,\F)}$ and $\nu\leq \sigma_{t-1}\leq \sigma_t$. Hence by (EP), there is $\tau\in \F$ such that both $P_1(\nu, \tau)$ and $P_1(\sigma_t,\tau)$ hold. But every $x\in \bigcup_{i=1}^{t-1} \sigma_i$ is a join of elements from $\nu$, so $P_1(\nu, \tau)$ implies $P_1(\sigma_{i}, \tau)$ for all $1\leq i\leq t-1$. 
\end{proof}

Using (EP), we can give bounds on the dimensions of $\red{\D(\S,\F)}$ and $\red{\PD(\S,\F)}$:

\begin{lemma}
\label{lem_dimensions}
Let $\S$ be a poset and $\F$ a non-empty set of frames in $\S$ that satisfies (EP). Assume that $m \coloneqq \sup \big\{\,|\tau|\, \tq \,\tau\in \F\,\big\}$ is finite. Then:
    \begin{enumerate}
        \item $0\leq \dim(\red{\D(\S,\F)})\leq m-1$.
        \item $m-1\leq \dim(\red{\PD(\S,\F)})\leq 2m-2$.
        \item If for all $\tau\in \F$, the poset $\bool(\tau)$ is the poset of non-empty faces of the boundary of an $(m-1)$-simplex, then $\dim(\red{\PD(\S,\F)}) = 2m-3$ and $\dim(\red{\D(\S,\F)}) = m-2$.
    \end{enumerate}
\end{lemma}
\begin{proof}
Let $\sigma_0<\cdots < \sigma_k$ be a chain in $\red{\PD(\S,\F)}$. Then by \ref{lm:EPproperty}, there is a frame $\tau\in \F$ such that for all $0\leq i \leq k$, every $x\in \sigma_i$ is a join of elements from $\tau$.
If the chain is maximal, then for all $i$, we either have
\begin{equation}
\label{eq_step_up_join}
    \sigma_{i+1} = \sigma_i\setminus\sigma'\cup \{ \bigvee_{y\in \sigma'} y\}
\end{equation}
for some $\sigma'\subseteq \sigma_i$ whose elements have a join, or
\begin{equation}
\label{eq_step_up_extend_partial_frame}
    \sigma_{i+1} = \sigma_i\cup \{ x\},
\end{equation}
where $x\in \tau$ such that $x\not \leq y$ for all $y\in \sigma_i$. 

There can be at most $(|\tau|-1)$-many occurrences of both Eq. (\ref{eq_step_up_join}) and Eq. (\ref{eq_step_up_extend_partial_frame}), where Eq. (\ref{eq_step_up_extend_partial_frame}) cannot occur in $\red{\D(\S,\F)}$.

A similar proof yields the consequences stated in item (3).
\end{proof}

\subsection{Group actions and statement of the main theorem}
If a group $G$ acts on the poset $\S$ as a group of poset automorphisms then we call $\S$ a $G$-poset. We write $x^g$ for the image of $x \in \S$ under the action of $g \in G$. For a subset $\tau$ of $\S$ and $g \in G$ we write 
$\tau^g$ for $\{~x^g~|~x \in \tau\,\}$. By elementary arguments it follows that 
$\Sigma(\tau)^g = \Sigma(\tau^g)$ and hence $\Sigma(\tau)$ and $\Sigma(\tau^g)$
are isomorphic. Since also the antichain property is preserved under the action of $G$, it follows that the set of all frames of a $G$-poset $\S$ is invariant under
$G$. If $\F$ is some set of frames invariant under $G$ then again elementary arguments show that $G$ acts on $\CB(\S,\F)$ and $\red{\PD(\S,\F)}$ as a group of poset automorphisms and hence both become $G$-posets.
Using this $G$-action, our main theorem is as follows.

\begin{theorem}
\label{thm:main}
  Let $G$ be a group, $\S$ a $G$-poset and let $\F$ be a non-empty collection of frames that satisfies (EP) and is invariant under the action of $G$.
  Then $\CB(\S,\F)$ is $G$-homotopy equivalent to $\red{\PD(\S,\F)}$.
\end{theorem}

\section{Examples}
\label{sec:examples}

Before we prove \ref{thm:main}, we analyze the scenarios from \ref{thm_hom_eq_cb_pd} and \ref{cor:remainingcases}. In particular, we show that in each case the condition (EP) from \ref{thm:main} is satisfied and hence \ref{thm_hom_eq_cb_pd} and \ref{cor:remainingcases} follow. 

If $X$ is an object such that we can naturally associate a poset $\S$ with a family of frames $\F$, we will usually write $\red{\D(X)}$, $\red{\PD(X)}$ and $\CB(X)$ instead of $\red{\D(\S,\F)}$, $\red{\PD(\S,\F)}$ and $\CB(\S,\F)$ respectively.

\subsection{Free modules over a ring \texorpdfstring{$R$}{R}}
\label{sec_ex_free_modules}
\medskip

This is a generalization of $\S(\ZZ^n)$ from \ref{ex_Zn} and \ref{ex_EP} in the previous section, and it is motivated by Rognes' original definitions.
Take $M=R^n$, a free module of rank $n$ over a commutative ring $R$ with unit, and let $\S(M)$ be the poset of non-zero proper summands of $M$.
If $N,N'\in \S(M)$ are such that $N+N'\in \S(M)$, then this sum agrees with the join $N\vee N' = N+N'$. Also if $N\cap N' \in \S(M)$, then this agrees with the meet $N \wedge N' = N\cap N'$. However, in general, neither $N+N'$ nor $N\cap N'$ need to be summands of $M$; for the case of $N+N'$, we saw this in \ref{ex_Zn}.

We apply our definitions to
$\S = \S(M)$. 
As our collection $\F = \F(M)$ of frames we take the collections 
$\tau = \{ N_1,\ldots,N_n \} \subseteq \S(M)$ of rank one free elements of $\S(M)$ such 
that $N_1\oplus \cdots \oplus N_n = M$.
In this case, for each $\tau\in \F$, the poset $\Sigma(\tau)$ 
is isomorphic to the poset of proper faces of an $(n-1)$-simplex. 
In other words, $\Delta\Sigma(\tau)$ is a Coxeter complex of type $\mathtt{A}_{n-1}$.
We write $\red{\PD(M)} = \red{\PD(\S,\F)}$ and $\CB(M) = \CB(\S,\F)$. These complexes agree with the ones defined in the introduction.
The poset $\S$ carries a natural action by $G = \GL_n(R)$ that preserves $\F$, so $\CB(M)$ and $\red{\PD(M)}$ are $G$-posets.

\subsection*{Fields}

If $R = \KK$ is a field and $V = \KK^n$ is a finite-dimensional vector space over $\KK$, then $\S(V)$ is the poset of subspaces different from $\{0\}$ and $V$ and it is the proper part of a lattice. Its order complex is the Tits building of type $\mathtt{A}_{n-1}$ over $\KK$ \cite{AB:Buildings}.
    Moreover, $\red{\PD(V)}$ is the poset of non-empty proper partial direct sum decompositions of $V$.
    (EP) holds for this family by the Basis Extension Theorem.
    
    For each frame $\tau$, the complex $\Delta \Sigma(\tau)$ gives an embedded apartment in the building $\Delta\S$. While $\tau$ ranges over $\F$, this recovers the complete system of apartments of $\Delta\S$. One can think of $\CB(V)$ as being constructed from $\Delta\S$ by completing each such apartment to a simplex of dimension $2^n-3$ spanned by the $2^n-2$ vertices of the apartment.
    
If $R = \GF{q}$ is a finite field, Theorem 1.2 from \cite{HHS} describes the $\GL_n(q)$-module structure of
$\widetilde{H}_{2n-3}(\red{\PD(\GF{q}^n)};\CC)$ 
in the following way.
Let $N$ be the normalizer of a cyclic Coxeter torus $T$ in $\GL_n(q)$ of order $q^n-1$. It
is well-known that $N$ is a semidirect product of an element $f$ of
order $n$ and the Coxeter torus $T$. 
Let $\theta_n : N \rightarrow \CC^ \ast$ be the representation that 
sends $f$ to $e^ {\frac{2\,\pi\, i}{n}}$ and $T$ to $1$.
Then:
\begin{equation}
    \label{eq:rep}
\widetilde{H}_{2n-3}(\CB(\GF{q}^ n);\CC) \overset{\text{\ref{cor:field}}}{\cong} \widetilde{H}_{2n-3}(\red{\PD(\GF{q}^ n)};\CC) \overset{\text{Theorem 1.2 \cite{HHS}}}{\cong} \theta_n\uparrow_N^ {\GL_n(q)}.
\end{equation}

\subsection*{Dedekind domains} More generally, if $R$ is a Dedekind domain and $M = R^n$, it is known that $\S(M)$ is isomorphic to $\S(M\otimes \KK_R)$, the poset of proper non-trivial subspaces of $M\otimes \KK_R \cong \KK_R^n$, where $\KK_R$ is the field of fractions of $R$ (see \cite[p.~3]{Charney}).
    The isomorphism is given by $V\in \S(M\otimes \KK_R)\mapsto V\cap M$. Since this isomorphism preserves intersections, the intersection of summands of $M$ is also a summand of $M$.
    
    As in the case where $R$ is a field, every frame $\tau\in \F(R^n)$ yields an embedded apartment in the $\mathtt{A}_{n-1}$ building $\S(\KK_R^n)$ of the field of fractions. However, while $\tau$ ranges over $\F(R^n)$, this does not give the complete system of apartments, but just the apartments that are \emph{integral} in the sense of \cite{CFP}.
    So $\CB(M)$ can be thought of as being constructed from $\Delta\S(\KK_R^n)$ by completing each integral apartment to a simplex of dimension $2^n-3$.

\medskip

To apply \ref{thm:main}, we need to verify whether the collection $\F$ satisfies (EP).
We will use the following properties for an arbitrary $R$-module $M$ over a commutative ring $R$:
\begin{enumerate}[label=(\roman*)]
    \item If $P\leq N\leq M$ are summands of $M$, then $P$ is a summand of $N$.
    \item If $N\leq M$ are free and $N$ is a summand of $M$ with a free complement, then any basis of $N$ extends to a basis of $M$.
    \item If $M$ is finitely generated and free, and $N\in \sigma\in \red{\PD(M)}$, then $N$ has a free complement in $M$.
\end{enumerate}

Recall that an $R$-module $P$ is stably free if there exist finitely generated free modules $N,M$ such that $P\oplus N = M$.
The following proposition relates stably free modules with (EP) in the case of commutative rings.

\begin{proposition}
\label{prop:EPringsWithFC}
The following are equivalent for a commutative ring $R$ (with unit):
\begin{enumerate}
    \item Every stably free module is free.
    \item If $N,M$ are finitely generated free modules and $N$ is a summand of $M$ then $M/N$ is free.
    \item (EP) holds for $\F(M)$ for any finitely generated free module $M$.
\end{enumerate}
\end{proposition}

\begin{proof}
It is straightforward to show that (1) implies (2).

We prove that (2) implies (3).
Let $M = R^n$ and take two partial decompositions $\sigma\leq \sigma'$ in $\red{\PD(M)}$.
We show that there exists a frame $\tau\in \F(M)$ such that each element of $\sigma\cup \sigma'$ is obtained as the sum of some elements of $\tau$.

Denote by $\Phi(\sigma')$ the span of the submodules of $\sigma'$.
By (2), $\Phi(\sigma')$ has a free complement since it is a free summand of $M$.
Therefore, by properties (ii) and (iii) above, it is enough to establish the result for the case $\Phi(\sigma') = M$.
Write $\sigma' = \{N_1,\ldots,N_r\}$ and let $\sigma_i = \{S\in \sigma\tq S\leq N_i\}$.
Then $T_i:=\Phi(\sigma_i)$ is a free summand of $N_i$ by (i) above.
By (2), $T_i$ has a free complement $P_i$ in $N_i$.
Thus we can take frames $\tau_i\in \F(T_i)$ and $\tau'_i\in \F(P_i)$, and it is not hard to show that $\tau = \bigcup_i \tau_i\cup \tau'_i$ is a frame in $\F(M)$.
This shows that (2) implies (3).

Finally, we show that (3) implies (1).
Assume that $P\oplus N = M$ where $N,M$ are finitely generated free modules.
We show that $P$ is free.
Let $Q = \tilde{N}\oplus M = \tilde{N}\oplus (N\oplus P)$, where $\tilde{N}\cong N$. This is a finitely generated free module. Both $N$ and $M$ are free summands of $Q$ with free complements $\tilde{N}\oplus P \cong M$ and $\tilde{N}$, respectively.
Hence $\{N\}$,$\{M\}\in \red{\PD(Q)}$ and $\{N\}\leq \{M\}$.
By (EP), there exists a basis $B$ of $Q$ such that $N$, $M$ are spanned by members of $B$.
Then $P = M/N$ is free with basis $B\cap M \setminus N$.
\end{proof}

According to Lam \cite{Lam}, a ring satisfying any of the equivalent properties of \ref{prop:EPringsWithFC} is called a Hermite ring.
The following proposition shows that various rings $R$ are Hermite rings.

\begin{corollary}
\label{prop:EP}
If $R$ is a local ring, a semilocal ring, a polynomial ring over a PID with a finite number of indeterminates, or a Dedekind domain, then $R$ is a Hermite ring, that is,~(EP) holds for $\F(R^n)$.
\end{corollary}

\begin{proof}
See Example 2.7 and Corollary 5.8 of \cite{Lam}.
\end{proof}

\ref{prop:EP} shows that Theorem~\ref{thm:main} implies \ref{thm_hom_eq_cb_pd}.


\subsection{Matroids}

Recall that a matroid is a pair $(M, \I(M))$ where $\I(M)$ is a pure simplicial complex of finite dimension $n-1$ whose set of vertices is $M$ and satisfies the following \textit{exchange property}: if $A,B$ are simplices of $\I(M)$ such that $|A|>|B|$ then there exists $x\in A\setminus B$ such that $B\cup \{x\}$ is also a simplex of $\I(M)$.
We will usually denote a matroid $(M,\I(M))$ just by $M$. Note that contrary to the usual convention 
we allow $M$ to be infinite.

For a subset $F\subseteq M$ we define its \emph{rank} by $\rk(F) = \max\{|S| \tq S\subseteq F$ and $S\in \I(M) \}$. In particular, the rank $n$ of $M$ is the size of a maximal simplex of $\I(M)$.
A \textit{flat} of the matroid $M$ is a subset $F\subseteq M$ such that $\rk(F \cup \{x\}) = \rk(F)+1$  for all $x\in M\setminus F$.
The set of all flats is denoted by $\L(M)$ and it is an atomistic and semimodular lattice (a geometric lattice when it is finite), whose dimension $n$ is the rank of $M$. The unique minimal element $0_{\L(M)}$ of $\L(M)$ is the flat whose
elements are called loops and $1_{\L(M)} = M$ is the unique maximal
element of $\L(M)$. 

We apply our definitions to
$\S = \L(M) \setminus \{ 0_{\L(M)}, 1_{\L(M)} \}$, whose height equals $n$.
As our collection $\F$ of frames we take the
sets $\tau$ of rank one flats for which $\Sigma(\tau)$ 
is the poset of proper faces of an $(n-1)$-simplex. 
It is then easily seen that a collection of rank one flats is a frame if and only if any choice of non-loop representatives from each flat yields a basis of $M$.

For this family $\F$ the definition of the poset of (direct sum) decompositions and partial decompositions of a matroid 
from \cite{PW,W} coincides with our definition (cf. \ref{rk_def_PW}).
Indeed, any decomposition in the sense of \cite{PW,W} refines into an $n$-frame.
Thus we just write $\red{\D(M)} = \red{\D(\S,\F)}$, $\red{\PD(M)} = 
\red{\PD(\S,\F)}$, and $\CB(M)=\CB(\S,\F)$.

Next we check that (EP) holds for matroids. Let $F \subseteq F'$ be two flats
and $B \subseteq F$, $B' \subseteq F'$ be independent sets such that $|B| = \rk(F)$ and $|B'| =\rk(F')$. 
Then, by the exchange property, there is $B \subseteq B'' \subseteq F'$ such that $|B''| = \rk(F')$. 
Let $\sigma, \sigma' \in \red{\PD(M)}$ with $\sigma \leq \sigma'$. Then using this principle we construct a basis $\bar{B}$ of $M$ such that $|\bar{B} \cap F| = \rk(F)$ for any flat $F$ in $\sigma\cup\sigma'$.
Now the frame $\tau$ determined by $\bar{B}$ satisfies $P_1(\sigma, \tau)$ and $P_1(\sigma',\tau)$. 
This shows that \ref{thm:main} implies the matroid case of \ref{cor:remainingcases}.

\begin{example}
Let $M$ be the matroid whose independent sets are the linearly independent sets of a finite-dimensional vector space $V$.
Then $\L(M)$ is the poset of subspaces of $V$.
\end{example}

\begin{example}
Let $M$ be the matroid such that $\I(M)$ is a simplex.
Then $\L(M)$ is the face poset of $\I(M)$, which is a Boolean lattice.
Here $\red{{\D}(M)}\cup \{ \{M\}\}$ is the partition lattice, and $\red{\PD(M)}$ is the proper part of the lattice of partial partitions.
Then $\red{\PD(M)}$ is contractible (see Corollary 6.24 of \cite{PW}).
\end{example}

It would be interesting to see if one can formulate an analogous version of Rognes' conjecture in other contexts, as for example for matroids, since $\red{\PD(M)}\simeq \CB(M)$ by \ref{thm:main}.
However, the following example shows that at least the reformulation of the conjecture we give in the introduction does not always carry over.

\begin{example}
\label{ex_uniform_matroid}
Let $M = U_{n,k}$ be the uniform matroid with $2\leq k \leq n-1$ whose 
proper flats are subsets of $[n]$ with at most $k-1$ elements.
That is, the proper part of $\L(M)$ is the face poset of the $(k-2)$-skeleton of an $(n-1)$-simplex.
Note $M$ has rank $k$.
Then $\red{\PD(U_{n,k})}$ has dimension $2k-3$, but it is homotopy equivalent to the proper part of $\L(U_{n,k+1})$, which is a wedge of spheres of dimension $k-1$.
See discussion on p.~41 of \cite{PW}.
\end{example}

\subsection{Non-degenerate subspaces}

Suppose that $V$ is a finite-dimensional vector space over a field $\KK$.
If $\sigma$ is an automorphism of $\KK$ of order at most $2$ and $\epsilon \in \{1,-1\}$,
then an $(\epsilon,\sigma)$-sesquilinear form is a bi-additive form $\Psi:V\times V\to \KK$ which is $\KK$-linear in the first variable and for all $v,w\in V$ it holds that
\[ \Psi(v,w) = \epsilon \sigma(\Psi(w,v)).\]
We further assume that if $\KK$ has characteristic $2$ and $\sigma$ is the identity, then $\Psi(v,v) = 0$ for all $v\in V$.

Let $V$ be equipped with a non-degenerate $(\epsilon,\sigma)$-sesquilinear or quadratic form $\Psi$, and take $\S = \S(V,\Psi)_{\operatorname{nd}}$ to be the poset of non-trivial proper non-degenerate subspaces of $V$, which is not the proper part of a lattice in general.
Here we consider orthogonal decompositions, so we let $\F$ be the collection of orthogonal frames.
These are \textit{minimal} orthogonal decompositions of $V$, that is, direct sum decompositions of $V$ into minimal non-degenerate subspaces of $V$ that are pairwise orthogonal. Clearly, for a frame $\tau$ we have that
$\Sigma(\tau)$ is the face poset of the boundary of a $(|\tau|-1)$-simplex. 
If $(V,\Psi)$ is not a symplectic space, then a minimal non-degenerate subspace of $V$ has dimension $1$; otherwise, a minimal non-degenerate subspace has dimension $2$.
Therefore, if we write $n = \dim(V)$ and $2n = \dim(V)$ in the symplectic case, then $\S(V,\Psi)_{\operatorname{nd}}$ has height $n$ and a frame has size $n$.
We write $\red{\D(V,\Psi)_{\operatorname{nd}}} = \red{\D(\S,\F)}$, $\red{\PD(V,\Psi)_{\operatorname{nd}}} =\red{\PD(\S,\F)}$ and $\CB(V,\Psi)_{\operatorname{nd}} = \CB(\S,\F)$.
The isometry group $G$ of $(V,\Psi)$ has a natural action on $\S(V,\Psi)_{\operatorname{nd}}$ that preserves $\F$ and hence these posets become $G$-posets.

We show that orthogonal complementation implies that our collection satisfies (EP).
To be more precise, assume we have two partial orthogonal decomposition $\sigma = \{S_1,\ldots,S_r\}$ and $\sigma' = \{T_1,\ldots,T_m\}$ such that $\sigma$ is finer than $\sigma'$.
Take an orthogonal frame $\sigma_i$ of each non-degenerate subspace $S_i$.
By orthogonality, it is clear that $\cup_i \sigma_i$ is a frame for the (orthogonal) sum $S_1\oplus \cdots \oplus S_r$.
Now, for each $1\leq j \leq m$, let $W_j$ be the orthogonal complement of the span of the $S_i$ contained in $T_j$:
\[ W_j = \gen{S_i\tq S_i\leq T_j}^\perp.\]
Similarly, we can take frames $\tau_j$ of each $W_j$, and by orthogonality we get a frame decomposition $\tau := \bigcup_j \tau_j \cup \bigcup_i \sigma_i$ of $T = T_1\oplus\cdots\oplus T_m$.
Finally, if $\rho$ is a frame of the orthogonal complement $T^\perp$, then $\tau\cup \rho$ is a frame of $V$ such that its elements span the $S_i$ and the $T_j$.
Note that these orthogonal complements are non-degenerate since $V$ and the elements of $\sigma\cup\sigma'$ are non-degenerate vector spaces.
This completes the verification of (EP) and shows that \ref{thm:main} implies the 
non-degenerate vector space case of \ref{cor:remainingcases}.

Finally, we mention that $\red{\PD(V,\Psi)}_{\operatorname{nd}}$ is indeed homotopy equivalent to $\red{\D(V,\Psi)}_{\operatorname{nd}}$ (see Corollary 6.17 of \cite{PW}).
In particular, $\red{\PD(V,\Psi)}_{\operatorname{nd}}$ is not spherical since it collapses to a smaller dimensional subposet that is not contractible in general (cf. Theorem 6.19 of \cite{PW}).
This shows that the analogue of Rognes' conjecture is not true in this context.

\subsection{Isotropic subspaces}
\label{sec_ex_isotropic_subspaces}
If $V = \KK^{2n}$ is equipped with a symplectic form $\Psi:V\times V\to \KK$, there is another natural choice for $\S$ and $\F$ that is closer to the building-like situation of \ref{sec_ex_free_modules}:
let $\S = \S(V,\Psi)_{\operatorname{ti}}$ be the poset of non-zero totally isotropic subspaces of $\KK^{2n}$. The order complex of $\S(V,\Psi)_{\operatorname{ti}}$ is the building of type $\mathtt{C}_n$ over $\KK$. It comes with a natural action by the symplectic group $G = \operatorname{Sp}_{2n}(\KK)$.
Let $\F$ be the collection of sets of $2n$ lines $\tau =\{L_1, \ldots, L_n, L_{-n}, \ldots, L_{-1}\}$ that are spanned by a symplectic basis of $\KK^{2n}$. That is, we have $\Psi(L_i, L_j) = \{0\}$ if and only if $i\neq -j$ (see \cite[Definition 2.8]{BH}).
The join in $\S(V,\Psi)_{\operatorname{ti}}$ is given by the common span of subspaces; this join only exists if such a span is totally isotropic. Hence, the subposet $\Sigma(\tau)\subseteq \S(V,\Psi)_{\operatorname{ti}}$ is given by all totally isotropic subspaces spanned by subsets of $\tau$. Such a $\tau$ is indeed a frame in the sense of \ref{def_frame} because the meet in $\S_{\operatorname{ti}}$ is given by intersection of subspaces and $\S_{\operatorname{ti}}$ is the proper part of a lattice. That $\F$ satisfies (EP) follows from the fact that every partial symplectic basis can be extended to a symplectic basis (this can be deduced from \cite[Lemma I.2.6]{MH}).
It is also easy to see that the set $\F$ is preserved by the action of $\operatorname{Sp}_{2n}(\KK)$ on $\S(V,\Psi)_{\operatorname{ti}}$.
We write $\red{\D(V,\Psi)_{\operatorname{ti}}} = \red{\D(\S,\F)}$, $\red{\PD(V,\Psi)_{\operatorname{ti}}} =\red{\PD(\S,\F)}$ and $\CB(V,\Psi)_{\operatorname{ti}} = \CB(\S,\F)$.

In contrast to the examples above, $\Sigma(\tau)$ is not the poset of non-empty faces of the boundary of a simplex. Instead, it is isomorphic to the poset of non-empty faces of the boundary of an $n$-dimensional cross polytope. This is the Coxeter complex of type $\mathtt{C}_n$ and $\Delta \Sigma(\tau)$ is the corresponding apartment spanned by the lines $L_1, \ldots, L_n, L_{-n}, \ldots, L_{-1}$ in the building $\Delta \S$.
Similarly to the case of the type $\mathtt{A}_n$ building described in \ref{sec_ex_free_modules}, the complex $\CB(\S,\F)_{\operatorname{ti}}$ is obtained by taking the vertex set of $\S$ and adding one simplex for every embedded apartment. This complex has dimension $3^n-2$. Our main result, \ref{thm:main}, says that it is homotopy equivalent to the complex $\red{\PD(\S,\F)}_{\operatorname{ti}}$, which here has dimension $4n-3$.
One can show that this dimension is not optimal in the sense that $\red{\PD(\S,\F)}_{\operatorname{ti}}$ is homotopy equivalent to a subposet of smaller dimension.

\subsection{Free groups}

Let $F_n$ be the free group of rank $n$ and $\FC_n$ the poset of non-trivial, proper free factors of $F_n$. The order complex of $\FC_n$ is the free factor complex introduced by Hatcher--Vogtmann \cite{HV1}.
It is a consequence of the Kurosh subgroup theorem that $\check{\FC_n}$ is a lattice. It comes with a natural action of $G = \operatorname{Aut}(F_n)$. For some elementary properties of free factors, see \cite[Section 2, Problems 30--37]{MKS}.

We apply our definitions to
$\S = \FC_n$. 
As our collection $\F$ of frames we take the
collection of rank one free factor decompositions as defined for example, by Hatcher--Vogtmann \cite{HV3}. That is, $\F$ consists of sets of rank one free factors $\tau = \{H_1,\ldots,H_n\}$ such that 
\[ F_n = \langle H_1,\ldots,H_n\rangle = H_1 * \cdots * H_n. \]
For each such $\tau$, the poset $\Sigma(\tau)$ 
is isomorphic to poset of proper faces of an $(n-1)$-simplex. 
The action of $\operatorname{Aut}(F_n)$ on $\FC_n$ preserves $\F$.
We denote the corresponding posets by $\red{\D(F_n)} = \red{\D(\S,\F)}$, $\red{\PD(F_n)} = \red{\PD(\S,\F)}$, and $\CB(F_n) = \CB(\S,\F)$.

The family $\F$ satisfies (EP).
Indeed, to establish (EP) it is enough to show that if
$\{H_1,\ldots,H_r\}\in \red{\PD(F_n)}$ is a partial decomposition where each $H_i$ is contained in a free factor $H$, then $\{H_1,\ldots,H_r\}\in \red{\PD(H)}$ (the argument then is as in \ref{ex_EP}).
That this is true follows from the fact that $\check{\FC_n}$ is a lattice and if $K\leq H\leq F_n$ are free factors of $F_n$ then $K$ is a free factor of $H$  (in our situation, we take $K$ to be the join of the $H_i$).
This shows that \ref{thm:main} implies the free group case of \ref{cor:remainingcases}.

\section{Proof of \texorpdfstring{\ref{thm:main}}{Theorem 2.9}}
\subsection{Outline}
We start with an outline of the proof of \ref{thm:main} for the case where $\S = \S(\ZZ^n)$ is the poset of non-trivial, proper direct summands of $\ZZ^n$ and $\F$ is the set of all collections $\{L_1, \ldots, L_n\}$ of lines that form a direct sum decomposition $\ZZ^n = L_1\oplus  \cdots \oplus L_n$.
We want to show that there is a $\GL_n(\ZZ)$-homotopy equivalence between $\CB = \CB( \S(\ZZ^n), \F)$ and $\red{\PD} = \red{\PD(\S(\ZZ^n), \F)}$. To do this, we define poset maps $u\colon \Delta \red{\PD} \to \CB$ and $m\colon \Delta\CB \to \red{\PD}$.

The map $u$ is defined as follows: an element of $\Delta \red{\PD}$ is a chain in $\red{\PD}$; that is, a collection $\sigma_0 < \cdots < \sigma_t$, where each $\sigma_i = \{A_1, \ldots, A_k\}$ is a set of direct summands such that $A_1\oplus \cdots \oplus A_k $ is again a direct summand of $\ZZ^n$.
This is the same as saying that the $A_i$ have pairwise trivial intersection and are basis compatible, that is,~there is a basis $B$ of $\ZZ^n$ such that $A_i \cap B$ is a basis of $A_i$ for all $i$. 
We define $u\big(\,\{\sigma_0, \ldots, \sigma_t\}\,\big)\coloneqq \bigcup_{i=0}^t \sigma_i$ as the set of all direct summands that show up in (at least) one of the $\sigma_i$.
To see that this union is an element of $\CB$, we have to verify that 
it is basis compatible.
By definition, there is an appropriate basis for each single $\sigma_i$.
Since the $\sigma_i$ are ordered by refinement, each $A \in \sigma_i$
is contained on some $A' \in \sigma_{i+1}$. Using this and the
extension property (EP) for free abelian groups, one can find an appropriate basis for $\bigcup_{i=0}^t \sigma_i$. 

To define $m$, first note that an element of $\Delta \CB$ is a chain of non-empty simplices $\sigma_0\subset \cdots \subset \sigma_t$ from the common basis complex. So each $\sigma_i = \{A_1, \ldots, A_k\}$ is a set of basis compatible direct summands in $\ZZ^n$ and $\sigma_i\subset \sigma_{i+1}$. From this chain, we want to obtain an element in $\red{\PD}$, that is,~a collection of basis-compatible direct summands that have pairwise trivial intersection. We do this in three steps: first, we consider for each $i$ the set of all non-trivial intersections of summands in $\sigma_i$, denoted by $\Cl(\sigma_i)$. Then,
for each $i$ we retain the inclusionwise minimal intersections, that is, the set $\min(\Cl(\sigma_i))$. Each $\min(\Cl(\sigma_i))$ is already a set of basis-compatible 
summands with pairwise trivial intersection. As a last step, we define
$m(\sigma_0\subset \cdots \subset \sigma_t)$ as the set of inclusionwise maximal elements in the union $\min(\Cl(\sigma_0)) \cup \cdots \cup \min(\Cl(\sigma_t))$. This assures that $m$ is indeed compatible with the partial orders on $\Delta \CB$ and $\red{\PD}$ and hence it is a poset map. 
The following diagram provides an example for the case $\ZZ^6 = \langle e_1, e_2, e_3, e_4, e_5, e_6\rangle$:
\begin{gather*}
\underbrace{
\xymatrixcolsep{4mm}
\xymatrix{
    \sigma_0 = \{\langle e_1\rangle, \langle e_1, e_2\rangle, \langle e_2, e_3\rangle,  \langle e_4, e_5\rangle\} \ar@{~>}[d] &\subset & \sigma_1 = \left\{{\begin{array}{c}  \langle e_1\rangle, \langle e_1, e_2\rangle, \langle e_2, e_3\rangle,\\ \langle e_2, e_3, e_4\rangle,  \langle e_4, e_5\rangle, \langle e_6\rangle  \end{array}} \right\} \ar@{~>}[d] \\
    \Cl(\sigma_0) = \left\{{\begin{array}{c} \langle e_1\rangle, \langle e_1, e_2\rangle, \langle e_2\rangle,\\ \langle e_2, e_3\rangle, \langle e_4, e_5\rangle  \end{array}} \right\} \ar@{~>}[d] & & \Cl(\sigma_1) =  \left\{{\begin{array}{c}  \langle e_1\rangle, \langle e_2\rangle, \langle e_1, e_2\rangle, \langle e_2, e_3\rangle, \\ \langle e_2, e_3, e_4\rangle, \langle e_4\rangle, \langle e_4, e_5\rangle, \langle e_6\rangle  \end{array}} \right\}\ar@{~>}[d] \\
    \min\Cl(\sigma_0) = \{\langle e_1\rangle, \langle e_2\rangle, \langle e_4, e_5\rangle\}& & \min\Cl(\sigma_1) = \{\langle e_1\rangle, \langle e_2\rangle, \langle e_4\rangle, \langle e_6\rangle \}
    }
    }\\
    \mapsto m(\{\sigma_0,\sigma_1\}) = \{\langle e_1\rangle, \langle e_2\rangle, \langle e_4, e_5\rangle, \langle e_6\rangle\}
\end{gather*}

We claim that $u$ and $m$ define homotopy equivalences. To see this, we will verify that the compositions
\begin{equation*}
    u \circ \Delta m \colon \Delta\Delta\CB \to \CB \quad\text{and} \quad m \circ \Delta u \colon \Delta\Delta \red{\PD} \to  \red{\PD}\
\end{equation*}
are homotopy equivalences.
On geometric realisations, the domains of these maps are iterated barycentric subdivisions. This makes it a bit hard to keep track of the images of general elements. However, the situation is easier for elements of the original posets $\red{\CB}\subseteq \Delta \Delta \red{\CB}$ and $\red{\PD} \subseteq \Delta \Delta \red{\PD}$. Here, the map $u \circ \Delta m$ acts as $\min(\Cl(\cdot))$ and $m \circ \Delta u$ acts as the identity:

If $\alpha \in \Delta\Delta\CB$, then it is a chain of chains of simplices in the common basis complex. The simplest case is when $\alpha$ is identified with a single element $\sigma\in \CB$ (so $\alpha= \{\{\sigma\} \}$ is a chain that has only one element and this one element is a chain that again has just one element). Here, $\Delta m(\alpha) = \{ \min(\Cl(\sigma))\}$ is the one-element chain consisting of the set of all minimal non-trivial intersections of elements of $\sigma$. The image $u(\Delta m(\alpha)) = \min(\Cl(\sigma))$ then is the set consisting of all these minimal intersections.
Similarly, the simplest type of $\beta \in \Delta\Delta \red{\PD}$ is   $\beta= \{\{\sigma\} \}$ with $\sigma \in \red{\PD}$. We have $\Delta u(\beta) = \{ \sigma\}$. But then $m(\Delta u(\beta)) = \min(\Cl(\sigma)) = \sigma$ because $\sigma$ is an element in $\red{\PD}$, so the only non-trivial intersections of its elements are the elements themselves.

\subsection{Detailed proof}
In what follows, we will not distinguish between a simplicial complex $K$ and its face poset (that is,~the poset of non-empty faces ordered by inclusion).
Recall that $\Delta \S$ denotes the order complex of a poset $\S$.
If $K$ is a simplicial complex then its order complex $\Delta K$ is the barycentric subdivision of $K$.
If $G$ is a group acting simplicially on $K$, we say that $K$ is a $G$-complex.
In that case, the face poset of $K$ is also a $G$-poset.
Analogously, if $\S$ is a $G$-poset then $\Delta \S$ is a $G$-complex.

For an order-preserving map $f : \S \rightarrow \T$ between posets, we write $\Delta f$ for the map
\[\Delta f : \left\{ \begin{array}{ccc} \Delta \S & \to & \Delta \T \\ \{ x_0 , \ldots , x_r \} & \mapsto & \{ f(x_0) , \ldots , f(x_r) \} \end{array} \right. .\]
Clearly, if $f$ is a $G$-equivariant map, then $\Delta f$ and its geometric realization are $G$-{e\-qui\-va\-riant}.

If $\S,\T$ are $G$-posets and $f,g:\S\rightarrow \T$ are two order-preserving $G$-equivariant maps such that $f(x)\leq g(x)$ for all $x\in \S$, then the geometric realizations of $f$ and $g$ are $G$-homotopy equivalent.

\bigskip

For a subset $\pi$ of a poset $\S$, $\max(\pi)$ denotes the set of maximal elements of $\pi$. Analogously defined is $\min(\pi)$.
The following is an immediate consequence of Quillen's Theorem A (see e.g.~\cite{Qui78}) and the equivariant Whitehead  Theorem.
    
\begin{proposition}
\label{prop:homotEquivPosetOrdComplex}
For a $G$-poset $\S$, the poset maps
  $$\max[\S]: \left\{ \begin{array}{ccc} \Delta \S & \to & \S \\ c & \mapsto & \max(c) \end{array} \right. \quad \text{ and } \quad 
  \min[\S]: \left\{ \begin{array}{ccc} \Delta \S & \to & \S^{\op} \\ c & \mapsto & \min(c) \end{array} \right.$$ 
  give rise to $G$-homotopy equivalences between $\Delta\Delta \S$ and $\Delta \S = \Delta \S^{\op}$.
\end{proposition}

From now on, we fix a $G$-poset $\S$ and a non-empty collection $\F$ of frames stable under the action of $G$.
For a subset $\sigma\subseteq \S$, we denote by $\Cl(\sigma)$ the set of all meets of non-empty subsets of $\sigma$ that exist in $\S$.

\begin{lemma}
\label{lem_P_1_for_closure}
    Let $\sigma\subseteq \S$ and let $\tau$ be a frame such that $P_1(\sigma,\tau)$ holds. Then $P_1(\Cl(\sigma),\tau)$ holds as well.
\end{lemma}
\begin{proof}
    This follows immediately because $\bool(\tau)$ is closed under taking (existing) meets in $\S$, as we observed after \ref{def_frame}.
\end{proof}

The next lemma is a simple consequence of \ref{lem_P_1_for_closure}.

\begin{lemma} 
  The map $$\Cl : \left\{ \begin{array}{ccc} \CB(\S,\F) & \to & \CB(\S,\F) \\
    \sigma & \mapsto & \Cl(\sigma) \end{array} \right.$$ is a well-defined inclusion-preserving $G$-equivariant map and satisfies
    \begin{itemize}
        \item $\Cl(\Cl(\sigma)) = \Cl(\sigma)$; and
        \item $\sigma \subseteq \Cl(\sigma)$.
    \end{itemize}
    That is, $\Cl$ is a closure operation on the face poset of $\CB(\S,\F)$. In particular, $\Cl$ is a $G$-homotopy equivalence
    onto its image. 
\end{lemma}
\begin{proof}
    That the map is well-defined, that is,~that the image of every $\sigma\in \CB(\S,\F)$ is again contained in $\CB(\S,\F)$, follows from \ref{lem_P_1_for_closure}.

    That it is inclusion-preserving and $\sigma \subseteq \Cl(\sigma)$ is immediate from the definitions and $\Cl(\Cl(\sigma)) = \Cl(\sigma)$ follows from the associativity of the meet on $\S$, see Eq. (\ref{eq_associativity_join}).
    Finally, it is $G$-equivariant since the action of $G$ preserves existing meets.

    For the homotopy equivalence, let $Y$ be the image of $\Cl$, and $i:Y\to \CB(\S,\F)$ the inclusion.
    Then $Y$ is a $G$-poset and $i$ is $G$-equivariant.
    Since $\Cl \circ i = \Id_Y$ and $i \circ \Cl \geq \Id_{\CB(\S,\F)}$, we see that $i$ and $\Cl$ are homotopy inverses of each other, and so $\Cl:\CB(\S,F)\to Y$ is a $G$-homotopy equivalence.
\end{proof}

To prove our main theorem, we will produce homotopy equivalences between suitable subdivisions of the order complexes of $\CB(\S,\F)$ and $\red{\PD(\S,\F)}$.
In the following lemma, we introduce the first map.

\begin{lemma}
\label{lem_well_defined_m}
  The following is a well-defined and order-preserving $G$-equivariant map
  $$m: \left\{ \begin{array}{ccc}
  \Delta \CB(\S,\F) & \to & \red{\PD(\S,\F)} \\
  c  & \mapsto & m(c) = \max\Big( \,\bigcup_{\sigma \in c} \min\big(\,\Cl(\sigma)\,\big)\,\Big)
  \end{array} \right. .$$ 
\end{lemma}

\begin{proof}
  If $c,c'\in \Delta \CB(\S,\F)$ are two chains with $c\subseteq c'$, then
  every maximal element of $\bigcup_{\sigma \in c} \min(\Cl(\sigma))$ is contained in a maximal element of $\bigcup_{\sigma \in c'} \min(\Cl(\sigma))$.
  This shows that if both $m(c)$ and $m(c')$ lie in $\red{\PD(\S,\F)}$ then $m(c) \leq m(c')$ in $\red{\PD(\S,\F)}$.
  
  Thus we must prove that for any $c \in \Delta \CB(\S,\F)$ we have $m(c) \in \red{\PD(\S,\F)}$.
  As $m(c)$ is clearly a non-empty subset of $\S$, we need to show that there is $\tau\in \F$ such that $P_1(m(c),\tau)$ and $P_2(m(c),\tau)$ hold.

  Let $\mu$ be the maximal element of $c$. 
  By definition, there is $\tau \in \F$ such that $P_1(\mu,\tau)$ holds, that is,~for all $y\in \mu$, we have $y\in \bool(\tau)$. Since $c$ is ordered by inclusion, we have $P_1(\sigma,\tau)$ for all $\sigma\in c$ and by \ref{lem_P_1_for_closure}, we then also have $P_1(\Cl(\sigma),\tau)$. In particular, this means that for all $y\in m(c)$, we have $y\in \bool(\tau)$, that is,~$P_1(m(c),\tau)$ holds.

 We now show $P_2(m(c),\tau)$, that is,~for every $x\in \tau$, there is at most one $y\in m(c)$ such that $x\leq y$. 
 Assume that $x\leq y_1$ and $x\leq y_2$ for $y_1, y_2 \in m(c)$. As $y_1,y_2\in \bool(\tau)$ and they have the common lower bound $x$, their meet $y_1\wedge y_2$ exists. 
 On the other hand, $y_1,y_2$ are maximal elements of $\bigcup_{\sigma \in c} \min(\Cl(\sigma))$, thus there are $\sigma_1,\sigma_2\in c$ such that $y_1\in \min(\Cl(\sigma_1))$ and $y_2\in \min(\Cl(\sigma_2))$. The chain $c$ is ordered by inclusion, so we can assume that $\sigma_1 \subseteq \sigma_2$.
  This means that $y_1 \in \Cl(\sigma_1)\subseteq \Cl(\sigma_2)$.
  Since $y_2$ is a minimal element of $\Cl(\sigma_2)$ and $y_1\wedge y_2\in \Cl(\Cl(\sigma_2)) = \Cl(\sigma_2)$, we must have $y_1\wedge y_2= y_2$, that is,~$y_1\geq y_2$.
  But $y_1$ and $y_2$ are also maximal elements of the set $\bigcup_{\sigma \in c} \min(\Cl(\sigma))$, so $y_1 = y_2$.

  Finally, taking $\max$, $\min$, unions and $\Cl$ are $G$-equivariant operations, so $m$ is a $G$-equivariant map.
\end{proof}

Under the additional assumption (EP) we now produce a map in the other direction such that
both compositions with $m$ are homotopy equivalences. 

\begin{lemma}
  The map $$u:\left\{ \begin{array}{ccc} \Delta \left(\, \red{\PD(\S,\F)}\,\right) & \to & \CB(\S,\F)\\
  \{\sigma_0 , \ldots , \sigma_t\} & \mapsto & \displaystyle{\bigcup_{i=0}^t \sigma_i} \end{array} \right.$$
  is well-defined if and only if the collection $\F$ satisfies (EP).
  
 Therefore, if (EP) holds, $u$ is an order-preserving $G$-equivariant map of posets.
\end{lemma}

\begin{proof}
The well-definition equivalence follows from \ref{lm:EPproperty}.
It is also clear that $u$ is order-preserving if (EP) holds.
\end{proof}

Note that by \ref{prop:EPringsWithFC}, if $\S$ is the poset of non-trivial proper free summands of $R^n$ for a commutative ring $R$ and $\F = \F(R^n)$ is as in \ref{sec_ex_free_modules}, then the above lemma states that
$u$ is well-defined (for all $n$) if and only if $R$ is a Hermite ring.

\medskip

We now show that 
\begin{equation*}
    u \circ \Delta m \colon \Delta\Delta\CB(\S,\F) \to \CB(\S,\F) \quad\text{and} \quad m \circ \Delta u \colon \Delta\Delta \left(\,\red{\PD(\S,\F)}\,\right) \to  \red{\PD(\S,\F)}
\end{equation*}
are homotopy equivalences by showing that they are homotopic to homotopy equivalences.
Since after taking geometric realizations, we have $|\Delta m| = |m|$ and $|\Delta u|=|u|$, this will allow us to conclude that both $u$ and $m$ are homotopy equivalences.

\begin{proposition}
\label{prop:homotopyEquivalences}
Suppose (EP) holds for $\F$.
The following hold:
\begin{enumerate}
    \item $u \circ \Delta m \leq \Cl\circ \max[\CB(\S,\F)]\circ \max[\Delta\CB(\S,\F)]$.
    \item $m \circ \Delta u \geq \min[\red{\PD(\S,\F)}] \circ \min[\Delta\left(\,\red{\PD(\S,\F)}\,\right)]$.
\end{enumerate}
In particular, $u \circ \Delta m$ and $m \circ \Delta u$ are $G$-homotopy equivalences, and hence $u$ and $m$ are $G$-homotopy equivalences.
\end{proposition}

\begin{proof}
We write down the composition for item (1). Let $\alpha\in  \Delta\Delta\CB(\S,\F)$.
Then
\begin{align*}
 u\left(\, \Delta m(\alpha) \, \right) & = u\Big(\, \big\{\, m(a) : a\in \alpha\,\big\}\,\Big)\\
 & = \displaystyle{u\Big(\,\big\{\, \max\big(\, \bigcup_{\sigma \in a} \min \Cl(\sigma)\, \big) : a\in \alpha\,  \big\}\,\Big)}\\
 & = \displaystyle {\bigcup_{a\in \alpha} \max\Big(\, \bigcup_{\sigma \in a} \min \Cl(\sigma)\, \Big)}\\
 & \subseteq \bigcup_{a\in \alpha} \bigcup_{\sigma \in a} \min \Cl(\sigma)\\
 & = \bigcup_{\sigma \in \max(\alpha)} \min \Cl(\sigma)\\
 & \subseteq \bigcup_{\sigma \in \max(\alpha)} \Cl(\sigma)\\
 & = \Cl\Big(\,\max[\CB(\S,\F)]\big(\,\max[\Delta\CB(\S,\F)](\alpha)\,\big)\,\Big). 
\end{align*}
This proves item (1).

Now we show that the inequality of item (2) holds.
Let $\beta \in \Delta\Delta \left(\,\red{\PD(\S,\F)}\,\right)$.
Then
\begin{equation*}
\displaystyle{m\left( \,\Delta u(\beta)\, \right) = m\Big(\,\big\{\, \bigcup_{\sigma\in b} \sigma\, :\, b\in \beta \,\big\}\,\Big) = \max\Big(\, \bigcup_{b\in \beta} \min\big(\, \Cl (\, \bigcup_{\sigma\in b} \sigma \,)\,  \big)\,  \Big)}.
\end{equation*}
Note that if $\sigma\in \red{\PD(\S,\F)}$ then $\Cl(\sigma) = \sigma$, and if $b\in \Delta\red{\PD(\S,\F)}$ then
\[\displaystyle{\min(b)\subseteq \min\Big(\, \Cl\big(\,\bigcup_{\sigma\in b} \sigma\,)\, \Big)},\]
and so
\[ \displaystyle{\bigcup_{b\in \beta} \min(b) \subseteq \bigcup_{b\in \beta}  \min\Big(\, \Cl \big(\, \bigcup_{\sigma\in b} \sigma \,\big)  \,\Big) }.\]
In particular, this shows that the maximal elements of the left-hand side are contained in the maximal elements of the right-hand side.
That is, the left-hand side is finer than the right-hand side.

The element $\beta$ consists of chains of partial decompositions. Let $\{ \min(b) : b\in \beta\}$ be the set of minimal elements of these chains. It is itself a chain of partial decompositions, so an element in $\Delta\red{\PD(\S,\F)}$. Its maximum is the minimal element of the shortest, that is,~minimal, chain in $\beta$.
Thus we have proved
\begin{align*}
\min[\red{\PD(\S,\F)}]\big(\,\min[\Delta\left(\,\red{\PD(\S,\F)}\,\right)]& (\beta)\,\big) = \max\Big(\,\big\{\, \min(b)\, :\, b\in \beta\,\big\}\,\Big) = \max\big( \bigcup_{b\in \beta} \min(b) \big)\\
 & \leq \max\big( \bigcup_{b\in \beta} \min\big( \Cl \left( \cup_{\sigma\in b} \sigma \right)  \big)  \big)
 = m \left( \Delta u(\beta)\right).
\end{align*}
It follows that $\min[\red{\PD(\S,\F)}] \circ \min[\Delta\left(\,\red{\PD(\S,\F)}\,\right)] \leq m \circ  \Delta u$.
\end{proof}

Now the proof of \ref{thm:main} is immediate.

\begin{proof}[Proof of \ref{thm:main}]
  By \ref{prop:homotopyEquivalences} any of the maps $u$ and $m$ provides a $G$-homotopy equivalence between 
  $\CB(\S,\F)$ and $\red{\PD(\S,\F)}$.
\end{proof}

\end{document}